\documentclass[11pt]{article}
\usepackage[margin=1in]{geometry}
\usepackage{curves}

\usepackage{amssymb,amsfonts,amsmath,amsthm,amscd,mathrsfs,dsfont,bbm}
\usepackage{graphicx,float,epsfig,psfrag}
\usepackage{wrapfig}
\usepackage{subfig}

\DeclareMathAlphabet{\mathpzc}{OT1}{pzc}{m}{it}

\newtheorem{propo}{Proposition}[section]
\newtheorem{lemma}[propo]{Lemma}
\newtheorem{definition}[propo]{Definition}
\newtheorem{corollary}[propo]{Corollary}
\newtheorem{thm}[propo]{Theorem}

\def\E{{\mathbb E}}
\def\prob{{\mathbb P}}
\def\eps{\varepsilon}

\def\ueps{\underline{\varepsilon}}
\def\de{{\rm d}}
\def\deps{\dot{\varepsilon}}
\def\cX{{\cal X}}
\def\cY{{\cal Y}}
\def\F{{\mathbb F}}
\def\cF{{\cal F}}
\def\1{{\mathds 1}}

\newcommand{\wQ}{\widetilde{Q}} 

\newcommand{\reals}{{\mathds R }}

\newcommand{\X}{\mathcal{X}} 
\newcommand{\Y}{\mathcal{Y}}

\newcommand{\mZ}{\mathbb{Z}}

\newcommand{\pp}{\mathbb{P}}

\newcommand{\e}{\varepsilon}

\begin{document}

%
%

\title{Conditional Random Fields, Planted Constraint Satisfaction, \\and Entropy Concentration
}

\author{Emmanuel Abbe\thanks{~Department of Electrical Engineering and Program in Applied and Computational Mathematics, Princeton University, Email: eabbe@princeton.edu}\;\;\; and\;\; 
Andrea Montanari\thanks{~Departments of  
Electrical Engineering and Statistics, Stanford University, Email: montanari@stanford.edu}
}

\date{}

\maketitle

\begin{abstract}
This paper studies a class of probabilistic models on graphs, where edge variables depend on incident node variables through a fixed probability kernel. The class includes planted constraint satisfaction problems (CSPs), as well as more general structures motivated by coding and community clustering problems. It is shown that under mild assumptions on the kernel and for sparse random graphs, the conditional entropy of the node variables given the edge variables concentrates around a deterministic threshold. 
This implies in particular the concentration of the number of solutions in a broad class of planted CSPs, the existence of a threshold function for the disassortative stochastic block model, and the proof of a conjecture on parity check codes. It also establishes new connections among coding, clustering and satisfiability.  
\end{abstract}
\thispagestyle{empty}

\newpage
\pagenumbering{arabic}

%
%

\section{Introduction}
This paper studies a class of probabilistic models on graphs encompassing models of statistical learning theory, coding theory, and random combinatorial optimization. Depending on the framework, the class may be described as a family of conditional random fields, memory channels, or planted constrained satisfaction problems. We start by providing motivations in the latter framework. 

Constrained satisfaction problems (CSPs) are key components in the theory of computational complexity as well as important mathematical models in various applications of computer science, engineering and physics. 
In CSPs, a set of variables $x_1,\dots,x_n$ is required to satisfy a collection of constraints involving each a subset of the variables. In many cases of interest, the variables are Boolean and the constraints are all of a common type:  
e.g., in $k$-SAT, the constraints require the OR of $k$ Boolean variables or their negations to be TRUE, whereas in $k$-XORSAT, the XOR of the variables or their negations must equal to zero.  
Given a set of constraints and a number of variables, the problem is to decide whether there exists a satisfying assignment. 
In random CSPs, the constraints are drawn at random from a given ensemble, keeping the constraint density\footnote{The ratio of the expected number of constraints per variables.} constant. 
In this setting, it is of interest to estimate the \emph{probability}
that a random instance is satisfiable. One of the fascinating
phenomenon occurring for random instances is the phase transition,
which makes the task of estimating this probability much easier in the
limit of large $n$. For a large class of CSPs, and as $n$ tends to infinity, the probability of being satisfiable tends to a step function, jumping from 1 to 0 when the constraint density crosses a critical threshold. For random $k$-XORSAT the existence of such a critical threshold is proved \cite{2-xorsat,dubois,cuckoo,pittel-xorsat}.  
For random $k$-SAT, $k \geq 3,$ the existence of a $n$-dependent
threshold is proved in \cite{friedgut}. 
However it remains open to show that this threshold converges when $n$ tends to infinity. Upper and lower bounds are known to match up to a 
term  that is of relative order $k\,2^{-k}$ as $k$ 
increases \cite{AchlioetalNature}. Phase transition phenomena in other types of CSPs are also investigated in \cite{AchlioptasTwoCol, MRT09,AchlioetalNature}

In planted random CSPs, a ``planted assignment'' is first drawn, and the constraints are then drawn at random so as to keep that planted assignment a satisfying one. Planted ensembles were investigated in \cite{barthel,haanpaa,AAAI,21,jia,barriers}, and at high density in \cite{altarelli,vilenchik,feige2006complete}. 
In the planted setting, the probability of being SAT is always one by
construction, and a more relevant question is to determine the actual
\emph{number} of satisfying assignments. One would expect that this
problem becomes easier in the limit of large $n$ due to an asymptotic
phenomenon. This paper shows that, indeed, a concentration phenomenon
occurs: for a large class of planted CSPs (including SAT, NAE-SAT and
XOR-SAT) the normalized logarithm of the number or satisfying
assignment concentrates around a deterministic number. Moreover, this
deterministic threshold  is $n$-independent. 

It is worth comparing the result obtained in this paper for planted
CSPs, with the one obtained in \cite{AMarxiv} for non planted CSPs. 
In that case, the number of solution is zero with positive probability
and therefore the logarithm of the number of solution does not have a
finite expectation. Technically, standard martingale methods does not
allow to prove concentration, even around an $n$-dependent threshold. 
In  \cite{AMarxiv} an interpolation method \cite{GuerraToninelliLimit} is used to prove the existence
of the limit of  a `regularized' quantity, namely the logarithm of the
number of solutions plus one, divided by the number of variables.
A technical consequence of this approach is that the concentration
of this quantity around a value that is independent of $n$ can only be
proved when the UNSAT probability is known to be $O(1/\log(n)^{1+\eps})$.

This paper shows that --in the planted case-- the concentration around
an $n$-independent value holds unconditionally and is exponentially
fast. We use again the interpolation technique
\cite{GuerraToninelliLimit,FranzLeone,FranzLeoneToninelli,PanchenkoTalagrand,BayatiInterpolation,AMarxiv}
but with an interesting twist. While in all the cited references, the
entropy (or log-partition function) is shown to be superaddittive, in
the present setting it turns out to be subaddittive.

Let us also mention that a fruitful  line of work has addressed the
relation between planted random CSPs and their non planted
counterparts in the satisfiable phase  \cite{barriers}, and in \cite{quiet1,quiet2}.
These papers show that, when the number of solutions is sufficiently
concentrated,  planting does not play a critical role in the model. 
It would be interesting to use these ideas to `export' the  concentration result obtained here  to non planted models.

In this paper, we pursue a different type of approach.
Motivated by applications\footnote{Planted models are also appealing to cryptographic
  application, as hard instances with known solutions provide good
  one-way functions.}, in particular in coding theory and community
clustering, we consider extensions of the standard planted
CSPs to a  setting allowing soft probabilistic constraints.  
Within this setting, the planted solution is an unknown vector to be
reconstructed, and the constraints are regarded as noisy observations
of this unknown vector.
For instance one can recover the case of planted random $k$-SAT as follows. Each clause is generated by selecting
first $k$ variable indices $i_1,\dots,i_k$ uniformly at random, representing the hyperedge of a random graph. Then a clause is drawn uniformly among
the ones that are satisfied by the variables 
$x_{i_1},\dots,x_{i_k}$ appearing in the planted assignment. The clause can hence be
regarded as a noisy observation of $x_{i_1},\dots,x_{i_k}$. More generally the formula can be seen as a noisy observation of the planted assignment.  

Our framework extends the above to include numerous examples from coding theory and
statistics. 
Within LDPC or LDGM  codes \cite{RiU05}, encoding is performed by evaluating the
sum of a random subset of information bits and transmitting it through
a noisy communication channel. The selection of the information bits is described by a graphs, drawn at random for the code construction, and the transmission of these bits leads to a noisy observation of the graph variables.  
Similarly, a community clustering block model
\cite{goldenberg2010survey} can
be seen as a random graph model, whereby each edge is a noisy
observation of the community assignments of the adjacent nodes. 
Definitions will be made precise in the next section. 

The conditional probability of the unknown  vector given the noisy
observations takes the form of a graphical model, i.e. factorizes
according to an hypergraph whose nodes correspond to variables and
hyperedges correspond to noisy observations. Such graphical  models have been
studied by many authors in machine learning \cite{lafferty2001conditional} under the name of
`conditional random fields', and in \cite{andrea2008estimating} in the context of LDPC and LDGM codes. The
conditional entropy of the unknown vector given the observations
is used here to quantify the residual uncertainty of the vector. This is equivalent to considering the mutual information between the node and edge variables.  
In such a general setting, we prove that the conditional entropy per
variable concentrates around a well defined deterministic limit. 
This framework allows a unified treatment of a large class of
interesting random combinatorial optimization problems, raises new
connections among them, and opens up to new models. We obtain in
particular a proof of a conjecture posed in \cite{kumar} on
low-density parity-check codes, and the existence of a threshold function for the disassortative stochastic block model \cite{decelle}.

\section{The model}
Let $k$ and $n$ be two positive integers with $n\geq k$.

\begin{itemize}
\item Let $V=[n]$ and $g=(V,E(g))$ be a hypergraph with vertex set $V$ and edge set $E(g) \subseteq E_k(V)$, where 
$E_k(V)$ denotes the set of all possible $\binom{n}{k}$ hyperedges of order $k$ on the vertex set $V$. We will often drop the term ``hyper''.
\item Let $\X$ and $\Y$ be two finite sets called respectively the input and output alphabets. 
Let $Q(\cdot | \cdot )$ be a probability transition function (or channel) from $\X^k$ to $\Y$, i.e., for each $u \in \X^k$, $Q(\cdot | u)$ is a probability distribution on $\Y$. 
\item To each vertex in $V$, we assign a node-variable in $\X$, and to each edge in $E(g)$, we assign an edge-variable in $\Y$. We define 
\begin{align}
P_g(y|x) \equiv  \prod_{I \in E(g)} Q(y_I|x[I]), \quad x \in \X^V, y \in \Y^{E(g)}, \label{direct}
\end{align}
where $y_I$ denotes the edge-variable attached to edge $I$, and $x[I]$ denotes the $k$ node-variables attached to the vertices adjacent to edge $I$.  This defines for a given hypergraph $g$ the probability of the edge-variables given the node-variables. 
\end{itemize}

The above is a type of factor or graphical model, or a planted constraint satisfaction problem with soft probabilistic constraints. 
For each $x \in \X^V$, $P_g(\cdot|x)$ is a product measure on the set of edge-variables. We call $P_g$ a {\bf graphical channel with graph $g$ and kernel $Q$}. We next put the uniform probability distribution on the set of node-variables $\X^V$, and define the a posteriori probability distribution (or reverse channel) by 
\begin{align}
R_g(x|y) \equiv \frac{1}{S_g(y)}  P_g(y|x) 2^{-n}, \quad x \in \X^V, y \in \Y^{E(g)}, \label{reverse}
\end{align}
where
\begin{align}
S_g(y) \equiv  \sum_{x \in \X^V} P_g(y|x) 2^{-n} \label{z-def}
\end{align}
is the output marginal distribution. 


We now define two probability distributions on the hypergraph $g$, which are equivalent for the purpose of this paper:
\begin{itemize}
\item  A sparse Erd\"os-R\'enyi distribution, 
where each edge is drawn independently with probability $p=\frac{\alpha n}{\binom{n}{k}}$, where $\alpha >0$ is the edge density. 
\item A sparse Poisson distribution, where for each $I\in E_k(V)$, a number of edges $m_I$ is drawn independently from a Poisson distribution of parameter $p=\frac{\alpha n}{\binom{n}{k}}$. Note that $m_I$ takes value in $\mZ$, hence $G$ is now a multi-edge hypergraph. To cope with this more general setting, we allow the edge-variable $y_I$ to take value in $\Y^{m_I}$, i.e., $y_I=(y_I(1),\dots,y_I(m_I))$, and define (with a slight abuse of notation) 
\begin{align}
Q(y_I|x[I])=\prod_{i=1}^{m_I} Q(y_I(i)|x[I]).
\end{align}
This means that for each $I$, $m_I$ i.i.d.\ outputs are drawn from the kernel $Q$. If $m_I=0$, no edge is drawn. We denote by $\mathcal{P}_k(\alpha,n)$ the above distribution on (multi-edge) hypergraphs.
\end{itemize}

Since $p=\frac{\alpha n}{\binom{n}{k}}$, the number of edges concentrates around its expectation given by $\alpha n$ and the two models are equivalent in the limit of large $n$ --- at least they are equivalent for the subsequent results. 

\section{Main Results}\label{main-sec}
We now define the conditional entropy between the node and edge variables. This is equivalent, up to a constant shift, to the mutual information between the node and edge variables. 
\begin{definition}
Let $X$ be uniformly drawn in $\X^n$, $G$ be a random sparse hypergraph drawn from the $\mathcal{P}_k(\alpha,n)$ ensemble independently of $X$, and $Y$ be the output of $X$ through the graphical channel $P_G$ defined in \eqref{direct} for a kernel $Q$.
We define 
\begin{align}
H_G^{(n)}(X|Y) &\equiv - 2^{-n} \sum_{x \in \X^V} \sum_{y \in \Y^{E(G)}} P_G(y|x)  \log R_G(x|y), \\
H^{(n)}(X|Y) &\equiv \E_G H_G^{(n)}(X|Y),
\end{align}
where $P_G$ and $R_G$ are defined in \eqref{direct} and \eqref{reverse} respectively. Note that $H_G^{(n)}(X|Y)$ is a random variable since $G$ is random, and for a realization $G=g$, $H_g^{(n)}(X|Y)$ is the conditional entropy of $X$ given $Y$, which can be expressed as $H_g^{(n)}(X|Y)=\E_y H_g^{(n)}(X|Y=y)$. Note that the mutual information between the node and edge variables is also obtained as $I_g^{(n)}(X|Y)=n-H_g^{(n)}(X|Y)$.

\end{definition} 

\begin{definition}
We denote by $M_1(\cX^l)$ the set of probability measures on $\cX^l$. For a kernel $Q$ from $\cX^k$ to $\cY$, we define  
\begin{align}
\Gamma_l: \, M_1(\cX^l) &\to \reals \\
\nu & \mapsto \Gamma_l(\nu)= \frac{1}{|\cY|} \sum_{u^{(1)}, \dots,u^{(l)} \in \cX^k} \left[ \sum_{z \in \cY}  \prod_{r=1}^l  (1-Q(z|u^{(r)}))  \right]  \prod_{i=1}^k \nu(u_i^{(1)},\dots,u_i^{(l)})  \label{gamma} \,.
\end{align}
\end{definition}

\noindent
{\bf Hypothesis H.} A kernel $Q$ is said to satisfy hypothesis H if $\Gamma_l$ is convex for any $l \geq 1$. \\

\noindent
Despite the lengthy expression, it is important to note that the definition of $\Gamma_l$ depends solely on the kernel $Q$. 
We will see in Section \ref{ex} that a large variety of kernels satisfy this hypothesis, including kernels corresponding to parity-check encoded channels, planted SAT, NAE-SAT, XORSAT, and disassortative stochastic block models. 

We first show a sub-additivity property for the expected conditional entropy of graphical channels. 
\begin{thm}\label{add}
Let $Q$ be a kernel satisfying hypothesis H and $n=n_1+n_2$, with $n_1,n_2 \geq k$. Then 
\begin{align}
H^{(n)}(X|Y) \leq H^{(n_1)}(X|Y) + H^{(n_2)}(X|Y).
\end{align}
\end{thm}
The proof of this theorem is outlined in Section \ref{interpolation-sec}.

\begin{corollary}\label{conv}
Let $Q$ be a kernel satisfying hypothesis H. There exists $C_k(\alpha,Q)$ such that 
\begin{align}
\frac{1}{n} H^{(n)}(X|Y) \to C_k(\alpha,Q), \quad \text{as $n \to \infty$}.
\end{align} 
\end{corollary}

The following is obtained using previous corollary and a concentration argument. 
\begin{thm}\label{conc}
Let $Q$ be a kernel satisfying hypothesis H, then, almost surely,
\begin{align}
\lim_{n\to\infty}\frac{1}{n} H_G^{(n)}(X|Y) = C_k(\alpha,Q),
\end{align} 
with $C_k(\alpha,Q)$ as in Corollary \ref{conv}. 
\end{thm}
The proof is given in Section \ref{main-proof}.


\section{Applications}\label{ex}
We next present three applications of the general model described in previous section.  
While planted CSPs and parity-check codes are directly derived as particular cases of our model, the stochastic block model is obtained with a limiting argument. One of the advantages of relying on a general model class, is that it allows to consider new hybrid structures. For example, one may consider codes which are not linear but which rely on OR gates as in SAT, or on community models whose connectivity rely on collections of $k$ nodes.    
\subsection{Planted constraint satisfaction problems}\label{csp-sec}


\begin{definition}\label{csp-def}
A CSP kernel is given by 
\begin{align}
Q(z|u) = \frac{1}{|A(u)|} \mathds{1} (z \in A(u)), \quad u\in \X^k, z \in \Y, \label{csp-kernel} 
\end{align}
where $A(u)$ is a subset of $\Y$ containing the ``authorized constraints'', with the property that $|A(u)|$ is constant (it may depend on $k$ but not on $u$). 
\end{definition}

We next show that a graphical channel with a CSP kernel corresponds to a planted CSP. 
We derive first a few known examples from this model.

\begin{itemize}
\item For planted $k$-SAT, $\Y=\{0,1\}^k$ and $A(u) = \{0,1\}^k \setminus  \bar{u}$, where $\bar{u}$ is the vector obtained by flipping each component in $u$. 
Using this kernel for the graphical channel means that for any selected edge $I \in E_k(V)$, the edge variable $y_I$ is a vector in $\{0,1\}^k \setminus  \bar{x}[I]$ uniformly drawn, representing the negation pattern of the constraint $I$. 
Note that using $u$ rather than $\bar{u}$ leads to an equivalent probabilistic model, $\bar{u}$ is simply used here to represent $u$ as a ``satisfying assignment''. Note that $|A(u)|=2^k-1$.  
\item For planted $k$-NAE-SAT, $\Y=\{0,1\}^k$ and $A(u) =\{0,1\}^k \setminus \{u, \bar{u}\}$, with $|A(u)|=2^k-2$. 
\item For $k$-XOR-SAT, $\Y=\{0,1\}$ and $A(u)= \oplus_{i=1}^k u_i$ and $|A(u)|=1$. 
%
\end{itemize}
In general, a graphical channel with graph $g$ and kernel $Q$ as in \eqref{csp-kernel} leads to a planted CSP where the constraints are given by $A(x[I]) \ni y_I$ for any $I \in E(g)$. For example, for planted $k$-SAT, the constraints are $\bar{x}[I] \neq y_I$, whereas for planted $k$-NAE-SAT, the constraints are $\bar{x}[I] \notin (y_I,  \bar{y}_I)$. If $y$ is drawn from the output marginal distribution $S_g$ (cf.\ \eqref{z-def}), then there exists a satisfying assignment by construction.

\begin{lemma}\label{counting-lemma}
For a graphical channel with graph $g$ and CSP kernel $Q$ as in \eqref{csp-kernel}, and for $y$ in the support of $S_g$, 
\begin{align}
H_g(X|Y=y) = \log Z_g(y)
\end{align} 
where $Z_g(y)$ is the number of satisfying assignments of the corresponding planted CSP.
\end{lemma}
\begin{corollary}
For a graphical channel with CSP kernel $Q$ as in \eqref{csp-kernel}, and for a graph drawn from the ensemble $\mathcal{P}(\alpha,n)$, 
\begin{align}
H^{(n)}(X|Y) = \E_{G,Y} \log Z_G(Y),
\end{align} 
where $Z_G(Y)$ is the number of satisfying assignments of the corresponding random planted CSP.
\end{corollary}

\begin{lemma}\label{k-sat-lemma}
For any $k\geq 1$, and for the CSP kernel corresponding to planted $k$-SAT, the operator $\Gamma_l$ is convex for any $l \geq 1$. 
\end{lemma}
\begin{lemma}\label{k-nae-sat-lemma}
For any $k\geq 1$, and for the CSP kernel corresponding to planted $k$-NAE-SAT, the operator $\Gamma_l$ is convex for any $l \geq 1$. 
\end{lemma}
\begin{lemma}\label{k-xor-sat-lemma}
For any $k$ even, and for the CSP kernel corresponding to planted $k$-XOR-SAT, the operator $\Gamma_l$ is convex for any $l \geq 1$. 
\end{lemma}


Using Theorem \ref{conc} and previous lemmas, the following is obtained.  
\begin{corollary}
For random planted $k$-SAT, $k$-NAE-SAT, and $k$-XOR-SAT ($k$ even), the normalized logarithm of the number of solutions concentrates in probability.  
\end{corollary}

\subsection{Stochastic block model}\label{sbm-sec}
The problem of community clustering is to divide a set of vertices in a network (graph) into groups having a higher  connectivity within the groups and lower connectivity across the groups (assortative case), or the other way around (disassortative case). This is a fundamental problem in many modern statistics, machine learning, and data mining problems with a broad range of applications in population genetics, image processing, biology and social science. 
A large variety of models have been proposed for community detection problems, we refer to \cite{newman1, fortunato,goldenberg2010survey} for a survey on the subject. 

At an algorithmic level, the problem of finding the smallest cut in a graph with two equally sized groups, i.e., the min-bisection problem, is well-known to be NP-hard \cite{dyer}. Concerning average-case complexity, various random graphs models have been proposed for community clustering.  
The Erd\"os-R\'enyi random graph is typically a very bad model for community structures, since each node is equally connected to any other nodes and no communities are typically formed. 
The stochastic block model is a natural extension of an Erd\"os-R\'enyi model with a community structure. Although 
the model is fairly simple (communities emerge but the average degree is still constant\footnote{Models with corrected degrees have been proposed in \cite{newman2}}), it is a fascinating model with several fundamental questions open. 

We now describe the stochastic block model (SBM), also called planted bisection model, with two groups and symmetric parameters. 
Let $V=[n]$ be the vertex set and $a,b$ be two positive real numbers.  
For a uniformly drawn assignment $X \in \{0,1\}^V$ on the vertices, an edge is drawn between vertex $i$ and $j$ with probability $a/n$ if $X_i=X_j$ and with probability $b/n$ if $X_i\neq X_j$, and each edge is drawn independently. We denote this model by $\mathcal{G}(n,a,b)$. 
Note that the average degree of an edge is $(a+b)/2$, however, a $0$-labelled node is connected in expectation with $a/2$ $0$-labeled nodes and with $b/2$ $1$-labeled nodes.

This type of model was introduced in \cite{dyer}, in the dense regime. The attention to the sparse regime described above is more recent, with \cite{coja-sbm} and \cite{newman2,decelle,mossel-sbm}. In particular, \cite{decelle} conjectured a phase transition phenomenon, with the detection of clusters\footnote{Obtaining a reconstruction positively correlated with the true assignment} being possible if $(a-b)^2 > 2(a+b)$ and impossible otherwise. 
In \cite{mossel-sbm}, a remarkable proof of the impossibility part is obtained, leaving the achievability part open. 

We next define a parametrized kernel which will allow us to approximate the above SBM model with a graphical channel. 
\begin{definition}
An SBM kernel is given by
\begin{align}
Q(z|u_1,u_2) =  \label{sbm-kernel}
\begin{cases}
a/\gamma & \text{ if } u_1 = u_2,\\
b/\gamma & \text{ if } u_1 \neq u_2,
\end{cases}
\end{align}
where $u_1,u_2,z \in \{0,1\}$. 
\end{definition}

\begin{lemma}\label{andrea-lemma}
There exists $n_0= n_0(\gamma,a,b)$ and $C= C(a,b)$ such that the
following holds true. 
Let $X$ be uniformly drawn on $\{0,1\}^V$, $Y$ be the output (the graph) of a
sparse stochastic block model of parameters $a,b$, and $Y_{\gamma}$ be
the output of a graphical channel with  graph ensemble
$\mathcal{P}(\gamma,n)$ and kernel \eqref{sbm-kernel}, then, for all
$n\ge n_0$
\begin{align}
\big|H^{(n)}(X|Y)-H^{(n)}(X|Y_{\gamma})\big|\le \frac{Cn}{\gamma} .
\end{align}
\end{lemma}

\begin{lemma}\label{annoying-lemma}
For the SBM kernel given by \eqref{sbm-kernel}, $a \leq b$ (disassortative case) and $\gamma$ large enough, the operator $\Gamma_l$ is convex for any $l \geq 1$. 
\end{lemma}

\begin{corollary}
For the disassortative SBM, the limit of $H(X^{(n)}|Y)/n$ exists and satisfies
\begin{align}
\lim_{n\to\infty}\frac{1}{n}H^{(n)}(X|Y) =
\lim_{\gamma\to\infty}\lim_{n\to\infty}\frac{1}{n}
H^{(n)}(X|Y_{\gamma})\, .
\end{align}
\end{corollary}
In a work in progress, the assortative case is investigated with a different proof technique. 
The computation of the above limit is also expected to reflect a phase transition for the SBM \cite{decelle,mossel-sbm}. 

\subsection{Parity-check encoded channels}\label{coding-sec}
The Shannon celebrated coding theorem states that for a discrete memoryless channel $W$ from $\X$ to $\Y$, the largest rate at which reliable communication can take place is given by the capacity $C(W)= \max_{X} I(X;Y)$, where $I(X;Y)$ is the mutual information of the channel $W$ with a random input $X$.  
To show that rates up to capacity are achievable, Shannon used random code books, relying on a probabilistic argument. 
Shortly after, Elias \cite{elias} showed that random linear codes allow to achieve capacity, reducing the encoding complexity from exponential to quadratic in the code dimension. However, Berlekamp, McEliece, and Van Tilborg showed in \cite{coding-np-hard} that the maximum likelihood decoding of unstructured linear codes is NP-complete. 

In order to reduce the complexity of the decoder, Gallager proposed the use sparse linear codes \cite{GallagerThesis}, giving birth to the LDPC codes, with sparse parity-check matrices, and LDGM codes, with sparse generator matrices. Various types of LDPC/LDGM codes depend on various types of row and column degree distributions. Perhaps one of the most basic class of such codes is the LDGM code with constant right degree, which corresponds to a generator matrix with column having a fixed number $k$ of one's. This means that each codeword is the XOR of $k$ uniformly selected information bits. In other words, this is a graph based code drawn from an Erd\"os-R\'enyi or Poisson ensemble $\mathcal{P}_k(\alpha,n)$. The dimension of the code is $m=\alpha n$ and the rate is $r=1/\alpha$. The code can also be seen as a planted $k$-XOR-SAT formula. 

Despite the long history of research on the LDPC and LDGM codes, and their success in practical applications of communications, there are still many open questions concerning the behaviour of these codes. In particular, even for the simple code described above, it is still open to show that the mutual information $\frac{1}{n}I(X^n;Y^m)$ concentrates, with the exception of the binary erasure channel for which much more is known \cite{Luby01,Luby97}. In the case of dense random codes, standard probability arguments show that concentration occurs with a transition at capacity for any discrete memoryless channels. But for sparse codes, the traditional arguments fail. Recently, the following was conjectured in \cite{kumar} for constant right degree LDGM codes and binary input symmetric output channels \footnote{This means that the channel is a weighted sum of binary symmetric channels}, 
\begin{align}
\pp\{ \frac{1}{m}I(X;Y) <  C(W)  \} \to  \label{shokro-conj}
\begin{cases}
0 & \text{ if } \alpha < C_k(W) \\
1 & \text{ if } \alpha > C_k(W) 
\end{cases}
\end{align}
where $C_k(W)$ is a constant depending on $k$ and $W$. 

We provide next a concentration result for this model, which implies the above conjecture for even degrees. 

\begin{definition}
An encoded symmetric kernel is given by 
\begin{align}
Q(z|u)= W(z | \oplus_{i=1}^k u_i),
\end{align}
where $W$ is a binary input symmetric output (BISO) channel from $\X$ to $\Y$.
\end{definition}
Note that this corresponds to the output of a BISO $W$ when the input to the channel is the XOR of $k$ information bits. 
This corresponds also to the constant right-degree LDGM codes considered in the conjecture of \cite{kumar}.


\begin{lemma}\label{bsc-lemma}
For an encoded symmetric kernel with $k$ even, the operator $\Gamma_l$ is convex for any $l \geq 1$. 
\end{lemma}

\begin{corollary}
Let $X$ be uniformly drawn in $GF(2)^n$, $U=XG$ be the output of a $k$-degree LDGM code of dimension $\alpha n$, and $Y$ be the output of $U$ on a BISO channel $W$. Then $\frac{1}{n}I(X;Y)$ converges in probability to a constant $C_k( \alpha, W)$. 
\end{corollary}
Note that $\frac{1}{m}I(X;Y) =  \frac{1}{m}H(Y)- H(W)$, where $H(W)$ denotes the conditional entropy of the channel $W$. 
Hence $\frac{1}{m}I(X;Y) < 1- H(W) \equiv \frac{1}{m}H(Y) < 1$. Since $\frac{1}{m}H(Y)$ converges from previous corollary, and since the limit must be decreasing in $\alpha$ (increasing in $r$), the conjecture \eqref{shokro-conj} follows.

\section{Proof outline for Theorem \ref{add}: Interpolation method for graphical channels}\label{interpolation-sec}
We now show the sub-additivity of $H^{(n)}(X|Y)$, 
\begin{align}
H^{(n)}(X|Y) \leq H^{(n_1)}(X|Y) + H^{(n_2)}(X|Y).
\end{align}
Note that if we partition the set of vertices $[n]$ into two disjoint sets of size $n_1$ and $n_2$ with $n_1+n_2=n$, and denote by $g_1$ and $g_2$ the restriction of $g$ onto these subsets obtained by {\bf removing} all the crossing hyperedges, then the following is obtained by basic properties of the entropy 
\begin{align}
H_g^{(n)}(X|Y) \leq H_{g_1}^{(n)}(X|Y) + H_{g_2}^{(n)}(X|Y). \label{fake}
\end{align}
Hence the above is true for a random graph $G$ drawn from the ensemble $\mathcal{P}_k(\alpha,n)$. However, the terms $H_{G_i}^{(n)}(X|Y)$, $i=1,2$, do not correspond to $H_{G_i}^{(n_i)}(X|Y)$, since the edge probability is $\frac{\alpha n}{\binom{n}{k}}$ and not $\frac{\alpha n_i}{\binom{n_i}{k}}$. Consequently, the above does not imply $H^{(n)}(X|Y) \leq H^{(n_1)}(X|Y) + H^{(n_2)}(X|Y)$. 
To obtain the proper term on the right hand side, one should add the edges lost in the splitting of the vertices (e.g., via a coupling argument), but this gives a lower bound on the right hand side of \eqref{fake}, conflicting with the upper bound. 
The interpolation method provides a way to compare the right quantities.

The interpolation method was first introduced in \cite{GuerraToninelliLimit}
for the Sherrington-Kirkpatrick model. This is a model for 
a spin-glass (i.e.\ a spin model with random couplings) on a complete
graph.
It was subsequently shown in 
\cite{FranzLeone,FranzLeoneToninelli,PanchenkoTalagrand}
that the same ideas can be generalized
to models on random sparse graphs, and applications in coding theory and random combinatorial optimization were proposed in \cite{Montanari05,Kudekar} and \cite{BayatiInterpolation,AMarxiv}. We next develop an interpolation method to estimate the conditional entropy of general graphical channels for different values of $n$. Interestingly, the planting flips the behaviour of the entropy from supper to sub-additive.

\begin{definition}
We define a more general Poisson model for the random graph, where a parameter $\e_I \geq 0$ is attached to each $I \in  E_k(V)$, and the number of edges $m_I(\e_I)$ is drawn from a Poisson distribution of parameter $\e_I$. This defines a random hypergraph whose edge probability is not homogenous but depends on the parameters $\e_I$. Denoting by $\underline{\e}$ the collection of all ${\binom{n}{k}}$ parameters $\e_I$, we denote this ensemble as $\mathcal{P}_k(\underline{\e},n)$. If for any $I$, $\e_I=\frac{\alpha n}{\binom{n}{k}}$, $\mathcal{P}_k(\underline{\e},n)$ reduces to $\mathcal{P}_k(\alpha,n)$ as previously defined.
\end{definition}

\begin{lemma}\label{derivative}
Let $X$ be uniformly drawn over $\X^n$, $G$ be a random hypergraph drawn from the ensemble $\mathcal{P}_k(\ueps,n)$  independently of $X$, and $Y(\ueps)$ be the output of $X$ through $P_G$ defined in \eqref{direct} for a kernel $Q$. Then
\begin{align}
\frac{\partial \phantom{\eps}}{\partial\eps_I} H^{(n)}(X|Y(\ueps)) =
-I(Y_I;X_I|Y(\ueps))\, ,
\end{align}
where $Y_I$ and $Y(\ueps)$ are independent conditionally on $X$ (i.e., $Y_I$ is drawn under $Q(\cdot|X[I])$ and $Y(\ueps)$ is drawn independently under $R_G(\cdot|X)$). 
 \end{lemma}

We define a \emph{path} as a differentiable map $t\mapsto \ueps(t)$, with $t\in[0,T]$ for some $T\geq 0$.
We say that a path is balanced if 
\begin{align}
\sum_{I\in E_k(V)}\frac{\de \eps_I}{\de t}(t) = 0\, .
\end{align}
We will write $\deps_I(t)$ for the derivative of $\eps_I(t)$ along the
path and  $Y(t)$ for $Y(\ueps(t))$. 
\begin{corollary} \label{coro:BalancedDerivative}
For a balanced path
\begin{align}
\frac{\de \phantom{t}}{\de t} H(X|Y(t)) =-\sum_{I\in E_k(V)}
H(Y_I|Y(t))\; \deps_I(t)\, .
\end{align}
\end{corollary}

Given a partition $V = V_1\sqcup V_2$, we define the
associated \emph{canonical path}  $\ueps : t \in [0, 1]\to \ueps(t) \in [0,1]^{E_k(V)}$ as
follows. Let $n_i = |V_i|$, $m_i=|E_k(V_i)|$, $i\in \{1,2\}$, and $m =
|E_k(V)|$. We define 
\begin{align}
\eps_I(0) \equiv \frac{\alpha n }{m}, \quad \forall I \in E_k(V),
\end{align}
\begin{align}
\eps_I(1)\equiv \begin{cases}
\frac{\alpha n_1 }{m_1} & \text{ if $I \in E_k(V_1)$} \\
\frac{\alpha n_2 }{m_2} & \text{ if $I \in E_k(V_2)$}\\
0 & \text{ otherwise.}  
\end{cases}
\end{align}
and 
\begin{align}
\ueps(t)= (1-t) \ueps(0) + t \ueps(1). 
\end{align}

%
Note that the canonical path is balanced. Moreover, at time $t=0$, $\mathcal{P}_k(\underline{\e}(0),n)$ reduces to the original ensemble $\mathcal{P}_k(\alpha,n)$, and at time $t=1$, $\mathcal{P}_k(\underline{\e}(1),n)$ reduces to two independent copies of the original ensemble on the subset of $n_1$ and $n_2$ variables: $\mathcal{P}_k(\alpha,n_1) \times \mathcal{P}_k(\alpha,n_2)$.

Applying Lemma \ref{coro:BalancedDerivative}, we obtain the following.
\begin{corollary} \label{coro:CanonicalDerivative}
For the canonical path
\begin{align}
\frac{\de \phantom{t}}{\de t} H(X|Y(t)) 
& = \alpha n
\E_I H(Y_I|Y(t)) 
-\alpha n_1 \E_{I_1}
H(Y_{I_1}|Y(t)) 
- \alpha n_2
\E_{I_2}H(Y_{I_2}|Y(t))
\, ,
\end{align}
where $I$ is drawn uniformly in $E_k(V)$, and $I_i$, $i \in \{1,2\}$, are drawn uniformly in $E_k(V_i)$. 
\end{corollary}
We recall that 
\begin{align}
H(Y_I|Y(t)) &= -  \E_{Y,Y_I} \log \sum_{x} Q(Y_I | x[I]) R_{G(t)}(x|Y)  \\
&= - \E_{Y(t),Y_I} \log  \E_{X|Y(t)} Q(Y_I | X[I]) \,,
\end{align}
where $Y(t)$ is the output of $P_{G(t)}$ and $\E_{X|Y(t)}$ is the conditional expectation over $R_{G(t)}$.


\begin{lemma}\label{replica}
\begin{align}
&\frac{1}{\alpha |\cY|}\frac{\de \phantom{t}}{\de t} H(X|Y(t)) = - \sum_{l=2}^\infty  \frac{1}{l(l-1)}  \E_{X^{(1)},\dots,X^{(l)}}   
\left[n \Gamma_l(V) - n_1 \Gamma_l(V_1) - n_2 \Gamma_l(V_2) \right]  \label{der}
\end{align}
where 
\begin{align}
\Gamma_l(V) \equiv   \E_{I,W_I} \prod_{r=1}^l \left(1-P(W_I | X^{(r)}[I]) \right), 
\end{align}
$I$ is uniformly drawn in $E_k(V)$, $W_I$ is uniformly drawn in $\Y$, and $X^{(1)},\dots,X^{(l)}$ are drawn under the probability distribution $\sum_{y} \prod_{i=1}^l R_{G(t)}(x^{(i)} | y) \sum_{u} P_{G(t)} (y|u) 2^{-n}$.
\end{lemma}

This means that $X^{(1)},\dots,X^{(l)}$ are drawn i.i.d.\ from the channel $R_{G(t)}$ given a hidden output $Y$, these are the `replica' variables, which are exchangeable but not i.i.d.. 
Note that denoting by $\nu$ the empirical distribution of $X^{(1)},\dots,X^{(l)}$, the above definition of $\Gamma_l(V)$ coincides with that of $\Gamma_l(\nu)$, hence the abuse of notation with definition \eqref{gamma}. Hypothesis H ensures that $\Gamma_l$ is convex for any distribution on $\X^l$, hence in particular for the empirical distribution of the replicas. 
Therefore, previous lemma implies Lemma \ref{add} and Theorem \ref{conv} follows by the sub-additivity property. 


%
\bibliographystyle{amsalpha}
\bibliography{gen}

%

\newpage
\appendix\label{app}
%
%
\section{Proofs of Theorems \ref{conv} and \ref{conc}}\label{main-proof}
We now prove the lemmas used in Section \ref{interpolation-sec} to prove Theorem \ref{conv}. We then prove Theorem \ref{conc}. In the proofs, we may drop the upper-script $n$ for the entropy.

\begin{proof}[Proof of Lemma \ref{derivative}]
Note that for a random variable $Z_\e$ which is Poisson distributed of parameter $\e$, and a function $f$,  
\begin{align}
\frac{\partial \phantom{\eps}}{\partial\eps} \E f(Z_\e) = \E f(Z_\e+1)- \E f(Z_\e).
\end{align}
Therefore, 
\begin{align}
\frac{\partial \phantom{\eps}}{\partial\eps_I} H^{(n)}(X|Y(\ueps)) &= H^{(n)}(X|Y(\ueps),Y_I)- H^{(n)}(X|Y(\ueps)),
\end{align}
where $Y_I$ is an extra output drawn independently from $Y(\ueps)$ but conditionally on the same $X$.  
We also recall the definition of the mutual information,
$I(A;B)=H(A)-H(A|B)=H(B)-H(B|A)$. 
We have
\begin{align}
\frac{\partial \phantom{\eps}}{\partial\eps_I} H(X|Y(\ueps)) &= - I(Y_I;X|Y(\ueps)) \\
&= -  H(Y_I |Y(\ueps)) +  H(Y_I| X,Y(\ueps))\\
&= -  H(Y_I |Y(\ueps)) +  H(Y_I| X_I,Y(\ueps)) \\
&= - I(Y_I;X_I|Y(\ueps)), 
\end{align}
where we used the fact that $Y_I$ depends only on the components of $X$ indexed by $I$. 
\end{proof}

\begin{proof}[Proof of Corollary \ref{coro:BalancedDerivative}]
Note that from previous proof, and using the fact that $Y_I-X-Y(\ueps)$ form a Markov chain, 
\begin{align}
\frac{\partial \phantom{\eps}}{\partial\eps_I} H(X|Y(\ueps)) &= - I(Y_I;X_I|Y(\ueps)) \\
&= -  H(Y_I |Y(\ueps)) +  H(Y_I| X_I) \\
&= -  H(Y_I |Y(\ueps)) +  H(Q) 
\end{align}
where $$H(Q) \equiv  - 2^{-k} \sum_{u \in \X^k, z \in \Y} Q(z|u) \log Q(z|u)$$ is a constant depending on $Q$ only. Therefore, if the path is balanced, the chain rule yields the result. 
\end{proof}

\begin{proof}[Proof of Lemma \ref{replica}]
By definition  
\begin{align}
H(Y_I|Y(t)) = - \E_{Y(t),Y_I} \log  \E_{X|Y(t)} Q(Y_I | X[I]) \,,
\end{align}
and expanding the logarithm in its power series, 
\begin{align}
\log  \E_{X|Y(t)} Q(Y_I | X[I]) &= - \sum_{l=1}^\infty \frac{1}{l}  (\E_{X|Y(t)}(1-Q(Y_I | X[I])))^l .
\end{align}
We now introduce the `replicas' $X^{(1)},\dots,X^{(l)}$, which are i.i.d.\ under $Q_{X|Y(t)}$, i.e., we have the Markov relation $X - Y(t) - (X^{(1)},\dots,X^{(l)})$. Denoting by $\wQ=1-Q$, we obtain 
\begin{align}
\log  \E_{X|Y(t)} Q(Y_I | X[I]) &=- \sum_{l=1}^\infty  \frac{1}{l}   \E_{X^{(1)},\dots,X^{(l)}|Y(t)} \prod_{r=1}^l \wQ(Y_I | X^{(r)}[I])).
\end{align}
As opposed to the non-planted case, where supper-additivity is achieved by showing that terms weighted by $1/l$ are convex in the empirical distribution of the replicas, convexity does not hold with the above expression and the following is needed.  Collecting terms we have 
\begin{align}
H(Y_I|Y(t)) 
&=   \E_X \E_{Y(t)|X} \E_{Y_I|X} \sum_{l=1}^\infty  \frac{1}{l}   \E_{X^{(1)},\dots,X^{(l)}|Y(t)} \prod_{r=1}^l \wQ(Y_I | X^{(r)}[I]))\\
&=   \sum_{l=1}^\infty  \frac{1}{l}   \E_{X,X^{(1)}\dots,X^{(l)}} \E_{Y_I|X} \prod_{r=1}^l \wQ(Y_I | X^{(r)}[I])).
\end{align}
We next write switch measure in the expectation $\E_{Y_I|X}$, defining $W_I$ to be uniformly distributed over $\Y$, and writing 
\begin{align}
H(Y_I|Y(t)) &= |\cY| \sum_{l=1}^\infty  \frac{1}{l}   \E_{X,X^{(1)}\dots,X^{(l)}} \E_{W_I} \prod_{r=1}^l \wQ(W_I | X^{(r)}[I])) Q(W_I | X[I])).
\end{align}
Renaming $X$ by $X^{(0)}$, and using the fact that $X^{(0)},X^{(1)}\dots,X^{(l)}$ are exchangeable, we can write 
\begin{align}
H(Y_I|Y(t))&=|\cY| \sum_{l=1}^\infty  \frac{1}{l}   \E_{X^{(0)},X^{(1)}\dots,X^{(l)}} \left( \E_{W_I} \prod_{r=1}^l \wQ(W_I | X^{(r)}[I])) -  \E_{W_I} \prod_{r=0}^l \wQ(W_I | X^{(r)}[I])) \right) \\
&= |\cY| \E_{X^{(1)}} \E_{W_I}  \wQ(W_I | X^{(1)}[I])) -  |\cY| 
\sum_{l=2}^\infty  \frac{1}{l(l-1)}  \E_{X^{(1)}\dots,X^{(l)}}  \E_{W_I} \prod_{r=1}^l \wQ(W_I | X^{(r)}[I])) .
\end{align}
Recall that from Corollary \ref{coro:CanonicalDerivative}, 
\begin{align}
\frac{\de \phantom{t}}{\de t} H(X|Y(t)) 
& = \alpha n
\E_I H(Y_I|Y(t)) 
-\alpha n_1 \E_{I_1}
H(Y_{I_1}|Y(t)) 
- \alpha n_2
\E_{I_2}H(Y_{I_2}|Y(t))
\, .\nonumber
\end{align}
Hence, carrying out the above expansions for each term, and since the term $\E_{X^{(1)}} \E_{W_I}  \wQ(W_I | X^{(1)}[I]))$ cancels out, the lemma follows.
\end{proof}

\begin{proof}[Proof of Theorem \ref{conc}]
Since we now that $H^{(n)}_G(X|Y)/n$ converges in expectation, it is
sufficient to show that it concentrates around its expectation. Indeed
we claim that there exists $B>0$ such that
\begin{align}
\prob\{|H^{(n)}_G(X|Y)-H^{(n)}(X|Y)|\ge n\Delta\}\le 2\,
e^{-nB\Delta^2}\, ,\label{eq:Concentration}
\end{align}
whence our thesis follows from Borel-Cantelli.

The proof of (\ref{eq:Concentration}) is a direct application of
Azuma-Hoeffding inequality, and we limit ourselves to sketching its
main steps.
We condition on the number $m=O(n)$ of hyperedges in $G$ and regard
$H_G^{(n)}(X|Y)$ as a function of the choice of the $m$
hyperedges. We claim that $|H_{G}^{(n)}(X|Y)-H_{G'}^{(n)}(X|Y) |\le
2C$, for some constant $C$, if $G$ and $G'$ differ only in one of their hyperedges, whence 
Eq.~(\ref{eq:Concentration}) follows by Azuma-Hoeffding.

In order to prove the last claim, let $G'=G+a$ denote the graph $G$ to
which hyperedge $a=(i_1,\dots,i_k)$ has been added. Then, writing
explicitly the component of $Y$ corresponding to hyperedge $a$, we
need to prove that
$|H_{G+a}^{(n)}(X|Y,Y_a)-H_G^{(n)}(X|Y)|\le C$. We have,
dropping the subscripts and superscript for the sake of simplicity,
\begin{align}
0\le H(X|Y)- H(X|Y,Y_a)&=H(X|Y) -H(X,Y_a|Y)+H(Y_a|Y) \\ &=
H(Y_a|Y)-H(Y_a|X,Y) \\& \le \log_2|\cY|\, ,
\end{align}
where the last inequality follows from the fact that $Y_a$ takes value in the finite set $\Y$.
\end{proof}
%
%
\section{Proofs of Lemmas \ref{counting-lemma}, \ref{k-sat-lemma}, \ref{k-nae-sat-lemma} and \ref{k-xor-sat-lemma} }\label{csp-proof}

\begin{proof}[Proof of Lemma \ref{counting-lemma}]
We have
\begin{align}
P_g(y|x) &= \prod_{I \in E(g)} Q(y_I|x[I]) \\
&= \prod_{I \in E(g)} \frac{1}{|A|}  \mathds{1} (y_I \in A(x[I]))\\
&=\begin{cases}
\frac{1}{|A|^{|E(g)|}} & \text{ if } x \sim y, \\
0 & \text{ otherwise, } 
\end{cases}
\end{align}
where $x \sim y$ means that $x$ is a satisfying assignment for $y$.  
Hence for a given $x$, $P_g(\cdot | x)$ is uniform on the set of all $y$'s verifying $x$, which has cardinality $|A|^{|E(g)|}$. 
Since $X$ is uniform, for a given $y$, $R_g(\cdot|y)$ is a uniform measure on a set of cardinality 
\begin{align}
\sum_{x \in \X^n} \prod_{I \in E(g)}  \mathds{1} (y_I \in A(x[I])) 
&= |\{x \in \X^n:  y_I \in A(x[I]), \forall I \in E(g) \}| = Z_g(y)
\end{align}
Therefore $H_g(X|Y=y)= \log Z_g(y)$.
\end{proof}

\begin{proof}[Proof of Lemma \ref{k-sat-lemma}]
For planted $k$-SAT, $\Y=\X^k=\{0,1\}^k$,  
\begin{align}
Q(z|u) &=  \frac{1}{2^k-1} \mathds{1} (z \neq  \bar{u}) 
\end{align}
and 
\begin{align}
 \Gamma_l(\nu)&= \frac{1}{2^k} \left(\frac{1}{2^k-1}\right)^l \sum_{u^{(1)}, \dots,u^{(l)} \in \cX^k} \left[ \sum_{z \in \cY}  \prod_{r=1}^l  \mathds{1}(  \bar{z} = u^{(r)})  \right]  \prod_{i=1}^k \nu(u_i^{(1)},\dots,u_i^{(l)}) \\
&= \frac{1}{2^k} \left(\frac{1}{2^k-1}\right)^l \sum_{u \in \cX^k}   \prod_{i=1}^k \nu(u_i,\dots,u_i) \\
&= \frac{1}{2^k} \left(\frac{1}{2^k-1}\right)^l \sum_{u_1,\dots,u_k \in \cX}   \prod_{i=1}^k \nu(u_i,\dots,u_i) \\
&= \frac{1}{2^k} \left(\frac{1}{2^k-1}\right)^l \left(\sum_{u_1 \in \cX}   \nu(u_1,\dots,u_1)\right)^k,
\end{align}
which is convex in $\nu$ for any $k,l \geq 1$.
\end{proof}

\begin{proof}[Proof of Lemma \ref{k-nae-sat-lemma}]
For planted $k$-NAE-SAT, $\Y=\X^k=\{0,1\}^k$, 
\begin{align}
Q(z|u) &=  \frac{1}{2^k-2} \mathds{1} (z \notin (u, \bar{u})) 
\end{align}
and 
\begin{align}
 \Gamma_l(\nu)&= \frac{1}{2^k} \left(\frac{1}{2^k-2}\right)^l \sum_{u^{(1)}, \dots,u^{(l)} \in \cX^k} \left[ \sum_{z \in \cY}  \prod_{r=1}^l  \mathds{1} (u^{(r)} \in (z,\bar{z}) )  \right]  \prod_{i=1}^k \nu(u_i^{(1)},\dots,u_i^{(l)}) \\
&= \frac{1}{2^k} \left(\frac{1}{2^k-2}\right)^l\sum_{b_1,\dots,b_l \in \X} \sum_{u \in \cX^k}   \prod_{i=1}^k \nu(u_i \oplus b_1,\dots,u_i\oplus b_l) \\
&= \frac{1}{2^k} \left(\frac{1}{2^k-2}\right)^l \sum_{b_1,\dots,b_l \in \X} \left(\sum_{u_1 \in \cX}   \nu(u_1 \oplus b_1,\dots,u_1 \oplus b_l)\right)^k,
\end{align}
which is convex in $\nu$ for any $k,l \geq 1$.
\end{proof}

\begin{proof}[Proof of Lemma \ref{k-xor-sat-lemma}]
This is a special case of Lemma \ref{fourier} for $s=1$, $d=-1$. 
\end{proof}
%
%
\section{Proofs of Lemmas \ref{andrea-lemma} and \ref{annoying-lemma} }\label{sbm-proof}
\begin{proof}[Proof of Lemma \ref{andrea-lemma}]
We introduce a new collection of random variables $\{Z_{ij}\}_{(ij)\in
  E_2(V)}$, taking values in $\{0,1,\ast\}$, and indexed by the $\binom{n}{2}$ edges of the complete graph
over vertex set $V$. These are conditionally independent given $X$
with distribution given as follows:
\begin{align}
Z_{ij} \big|_{X_i=X_j}&= \begin{cases}
1 & \mbox{with probability }a/n,\\
0 & \mbox{with probability }(2\gamma-a)/n,\\
\ast & \mbox{with probability }1-2\gamma/n,\\
\end{cases}\\
Z_{ij} \big|_{X_i\ne X_j}&= \begin{cases}
1 & \mbox{with probability }b/n,\\
0 & \mbox{with probability }(2\gamma-b)/n,\\
\ast & \mbox{with probability }1-2\gamma/n,\\
\end{cases}
\end{align}
The following claim is proved below. 
\begin{lemma}\label{lemma:Coupling}
There exists a constant $C = C(\gamma)<\infty$ such that,
uniformly in $n$
\begin{align}
\big|H^{(n)}(X|Y_{\gamma}) -H^{(n)}(X|Z)\big|\le C(\gamma)\, .
\end{align}
\end{lemma}
It is therefore sufficient to bound the difference
$|H^{(n)}(X|Y)-H^{(n)}(X|Z)|$. Notice that the variables $X,Y,Z$ can be
constructed on the same probability space in such a way that
$X-Z-Y$ form a Markov chain. Namely, it is sufficient to let the
$Y_{ij}$ be conditionally independent, and independent from $X$ given
$Z$, with 
\begin{align}
Y_{ij} = \begin{cases}
1 & \mbox{ if }Z_{ij} = 1,\\
0 & \mbox{ if }Z_{ij} \in \{0,\ast\}.\\
\end{cases}
\end{align}
We therefore have that $H^{(n)}(X|Z)\le H^{(n)}(X|Y)$ and
we are therefore left with the task of upper bounding
$H^{(n)}(X|Y)-H^{(n)}(X|Z)$.  In the rest of this proof we will omit
the superscript $(n)$ as it is will be fixed throughout.

We have, by the Markov property and the chain rule of conditional entropy
\begin{align}
H(X|Z) & = H(X|Y,Z)\\
& = H(X,Z|Y)-H(Z|Y)\\
& = H(X|Y) +H(Z|X,Y)- H(Z|Y)\, ,
\end{align}
and therefore
\begin{align}
H(X|Y)-H(X|Z) = H(Z|Y)-H(Z|X,Y)\, .
\end{align}
Note that, by subaddittivity of the entropy, and since conditioning
reduces entropy, we have
\begin{align}
H(Z|Y)\le \sum_{(i,j)\in E_2(V)} H(Z_{ij}|Y) \le \sum_{(i,j)\in E_2(V)} H(Z_{ij}|Y_{ij}) \, .
\end{align}
Further, by conditional independence of the $\{Z_{ij}\}$ given $X$,
$Y$, we have
\begin{align}
H(Z|X,Y) = 
 \sum_{(i,j)\in E_2(V)} H(Z_{ij}|X,Y) = \sum_{(i,j)\in E_2(V)} H(Z_{ij}|X_i,X_j,Y_{ij}) \, .
\end{align}
We therefore conclude that 
\begin{align}
H(X|Y)-H(X|Z) \le\sum_{(i,j)\in E_2(V)} \big\{H(Z_{ij}|Y_{ij})-H(Z_{ij}|X_i,X_j,Y_{ij})\big\}\, .
\end{align}
A simple calculation yields $H(Z_{ij}|Y_{ij}) = f_n((a+b)/2)$ and
$H(Z_{ij}|X_i,X_j,Y_{ij}) = f_n(a)+f_n(b))/2$, where
\begin{align}
f_n(c) \equiv -\frac{2\gamma-c}{n}\log \Big(\frac{2\gamma-c}{n}\Big) -
\Big(1-\frac{2\gamma}{n}\Big) \log\Big(1-\frac{2\gamma}{n}\Big)
+\Big(1-\frac{c}{n}\Big)\log\Big(1-\frac{c}{n}\Big)\, ,
\end{align}
Subtracting an affine term, we can write
\begin{align}
f_n(c) &= f_{n,0} + f_{n,1}c + g_n(c)\, ,\\
g_n(c) & =
-\frac{2\gamma}{n}\Big[\Big(1-\frac{c}{2\gamma}\Big)\log\Big(1-\frac{c}{2\gamma}\Big)+\frac{c}{2\gamma}\Big]
+\Big[\Big(1-\frac{c}{n}\Big)\log\Big(1-\frac{c}{n}\Big)+\frac{c}{n}\Big]\, .
\end{align}
Note that, for $x\in[-1/2,41/2]$, we have $0\le (1-x)\log(1-x)+x\le
x^2$. Hence, for $\gamma\ge c$, $n\ge 2c$, 
\begin{align}
0\ge g_n(c)\ge \frac{c^2}{n} \Big[\frac{1}{n}-\frac{1}{2\gamma}\Big]\ge
-\frac{c^2}{2n\gamma}\, ,
\end{align}
and therefore
\begin{align*}
H(X|Y)-H(X|Z) &\le \binom{n}{2}\Big[
f_n\Big(\frac{a+b}{2}\Big) - \frac{1}{2}f_n(a)-\frac{1}{2}f_n(b)
\Big]\\
& = \binom{n}{2}\Big[
g_n\Big(\frac{a+b}{2}\Big) - \frac{1}{2}g_n(a)-\frac{1}{2}g_n(b)
\Big]\\
&\le \frac{n^2}{2}\, \frac{1}{2n\gamma}\, \frac{(a-b)^2}{4} =
\frac{n(a-b)^2}{16\gamma}\, .
\end{align*}
This finishes the proof.
\end{proof}

\begin{proof}[Proof of Lemma \ref{lemma:Coupling}]
We write $Y_{\gamma} = \{Y_{\gamma,ij}\}_{(i,j)\in E_2(V)}$ where, for
each $(i,j)\in E_2(V)$, $Y_{\gamma,ij}$ is a vector containing
Poisson$(n\gamma/\binom{n}{2})$ entries, each being an independent
output of the channel $Q$ in Eq.~(\ref{sbm-kernel}) on input $(X_i,X_j)$. Analogously, we can
interpret $Z_{ij}$ as a vector of length Bernoulli$(2\gamma/n)$, with
the length $0$ corresponding to the value $\ast$. When the length of
the vector is equal to one, its entry is distributed as the output of
the same channel $Q$.

Let $\ell(Y_{\gamma,ij})$ and $\ell(Z_{ij})$ denote the length of
vectors $Y_{\gamma,ij}$ and $Z_{ij}$. It follows from standard
estimates on Poisson random variables that $Y$ and $Z$ can be coupled
in such a way that $\E \{|\ell(Y_{\gamma,ij})-\ell(Z_{ij})|\}\le C/n^2$
with $C=C(\gamma)$ and further, whenever $\ell(Z_{ij}) =1$ and
$\ell(Y_{\gamma,ij})\ge 1$, the first entry of the vector
$Y_{\gamma,ij}$ is equal to the only entry in $Z_{ij}$.

Finally notice that 
\begin{align}
H(X|Y_{\gamma},\ell(Y_{\gamma,ij})&=\ell_0+1)-H(X|Y_{\gamma},\ell(Y_{\gamma,ij})=\ell_0)\\
&=
H(Y'_{ij}|X,\ell(Y_{\gamma,ij})=\ell_0)-
H(Y'_{ij}|Y_{\gamma},\ell(Y_{\gamma,ij})=\ell_0)\, ,
\end{align}
with $Y'_{ij}$ distributed as an independent output of the channel $Q$
on input $X_i,X_j$, It follows that
 $|H(X|Y_{\gamma},\ell(Y_{\gamma,ij})=\ell_0+1)-H(X|Y_{\gamma},\ell(Y_{\gamma,ij})=\ell_0)|\le
 1$ and therefore
\begin{align}
|H(X|Y_{\gamma})-H(X|Z)|\le \sum_{(i,j)\in E_2(V)} \E
\{|\ell(Y_{\gamma,ij})-\ell(Z_{ij})|\}\le n^2\, \frac{C}{n^2}\le C\, .
\end{align}
\end{proof}

\begin{proof}[Proof of Lemma \ref{annoying-lemma}]
This is a special case of Lemma \ref{fourier} below, with $s=(a+b)/\gamma$ and $d=(a-b)/\gamma$. If $\gamma$ is large enough, $s \leq 1$ and if $a \leq b$, $d \geq 0$, and all the coefficients in \eqref{fourier-formula} are positive.  
\end{proof}

\begin{lemma}\label{fourier}
If $\cX=\cY=\{0,1\}$, $Q(y|x_1,\dots,x_k)=W(y|\oplus_{i=1}^k x_i)$ and $W$ is an arbitrary binary input/output channel, then
\begin{align}
\Gamma_l(\nu) = \frac{1}{2} \sum_{w \in \F_2^l} d^{|w|} \left[  s^{l-|w|} + (-1)^{|w|} (2-s)^{l-|w|} \right] \cF(\nu)^k(w)
\end{align}
where $s=W(1|0)+W(1|1)$, $d=W(1|0)-W(1|1)$, $|w|=\sum_{i=1}^l w_i$ and $\cF(\nu)(w)=\sum_{x \in \F_2^l} (-1)^{x \cdot w} \nu(x)$ is the Fourier-Walsh transform of $\nu$ (where $x \cdot w$ denotes the dot product of $x$ and $w$). 
\end{lemma}
Note that $ \cF(\nu)(w)$ is linear in $\nu$, hence  
\begin{itemize}
\item For $s=1$, i.e., for symmetric channels, 
\begin{align}
\Gamma_l(\nu) =  \sum_{w \in \F_2^k: \, |w| \text{ even}} d^{|w|}  \cF(\nu)^k(w)
\end{align} 
and $\Gamma_l$ is convex when $k$ is even. 
\item If $s \geq 1$, $d \geq 0$ or $s \leq 1$, $d \leq 0$, then $\Gamma_l$ is convex when $k$ is even. 
\end{itemize}

\begin{proof}[Proof of Lemma \ref{fourier}]
We have
\begin{align}
\Gamma_l(\nu)&= \frac{1}{2} \sum_{u^{(1)}, \dots,u^{(l)} \in \F_2^k} \left[ \sum_{y \in \F_2}  \prod_{r=1}^l  (1-P(y|u^{(r)}))  \right]  \prod_{i=1}^k \nu(u_i^{(1)},\dots,u_i^{(l)})
\end{align}
and using the fact that $P(y|u^{(r)}) = W(y|\oplus_{i=1}^k u_i^{(r)})$ ,
\begin{align}
\Gamma_l(\nu)&=\frac{1}{2}  \sum_{v^{(1)}, \dots,v^{(l)} \in \F_2} \left[ \sum_{y \in \F_2}  \prod_{r=1}^l  (1-W(y|v^{(r)}))  \right]   \nu^{\star k}(v^{(1)},\dots,v^{(l)})\\
&=\frac{1}{2}  \sum_{v \in \F_2^l}  \gamma(v)  \nu^{\star k}(v)
\end{align}
where 
\begin{align}
\gamma(v) &\equiv  \sum_{y \in \F_2}  \prod_{r=1}^l  (1-W(y|v^{(r)}))= (1-a)^{l-|v|} (1-b)^{|v|} + a^{l-|v|} b^{|v|} 
\end{align}
and $a=W(1|0)$, $b=W(1|1)$. 
Note that 
\begin{align}
a^{l-|v|} b^{|v|} \quad \stackrel{\cF}{\longleftrightarrow} \quad (a+b)^{l-|w|} (a-b)^{|w|}, \label{pair}
\end{align}
hence 
\begin{align}
\cF(\gamma)(w) = (2-(a+b))^{l-|w|} (a-b)^{|w|} (-1)^{|w|} + (a+b)^{l-|w|} (a-b)^{|w|} \,. \label{fourier-formula}
\end{align}
\end{proof}
\begin{proof}[Proof of \eqref{pair}]
To show that 
\begin{align}
\F_2^l \ni v \mapsto \rho^{|v|}   \quad \stackrel{\cF}{\longleftrightarrow} \quad \F_2^l \ni w \mapsto (1+\rho)^{l-|w|} (1-\rho)^{|w|}
\end{align}
note that the identity is true when $l=1$ and assume it to be true for $l$.
Then for $l+1$
\begin{align}
\sum_{v \in \F_2^{l+1}} \rho^{|v|} (-1)^{|vw|} &= \sum_{v \in \F_2^{l}} \rho^{|v|} (-1)^{|v w_1^l|} +\rho^{|v|+1} (-1)^{|v w_1^l|} (-1)^{w_{l+1}}  \\
&= \sum_{v \in \F_2^{l}} \rho^{|v|} (-1)^{|v w_1^l|} (1+\rho (-1)^{w_{l+1}}) \\
&= (1+\rho)^{l-|w_1^l|} (1-\rho)^{|w_1^l|}  (1+\rho (-1)^{w_{l+1}}) \,.
\end{align}
\end{proof}

%
%
\section{Proofs of Lemma \ref{bsc-lemma} }\label{coding-proof}

\begin{proof}
We represent the channel $W$ as a $2 \times |\Y|$ stochastic matrix. 
By definition of BISO channels, this matrix can be decomposed into pairs of columns which are symmetric as 
\begin{align}
\begin{pmatrix}
c & d \\
d & c
\end{pmatrix}
\end{align}
with $c,d\geq 0$, 
or into single columns which have constant values. Let us assume that $W$ contains $m$ such matrices and $s$ such constant columns. 
We have 
\begin{align}
\Gamma_l(\nu)&= \frac{1}{|\Y|} \sum_{u^{(1)}, \dots,u^{(l)} \in \F_2^k} \left[ \sum_{y \in \Y}  \prod_{r=1}^l  (1-P(y|u^{(r)}))  \right]  \prod_{i=1}^k \nu(u_i^{(1)},\dots,u_i^{(l)})\\
&=\frac{1}{|\Y|}  \sum_{v^{(1)}, \dots,v^{(l)} \in \F_2} \left[ \sum_{y \in \Y}  \prod_{r=1}^l  (1-W(y|v^{(r)}))  \right]   \nu^{\star k}(v^{(1)},\dots,v^{(l)})\\
&=\frac{1}{|\Y|}  \sum_{v \in \F_2^l}  g(v)  \nu^{\star k}(v)\\
&=\frac{1}{|\Y|}  \sum_{w \in \F_2^l}  \cF(g)(w)  \cF(\nu)^k(w) 
\end{align}
where 
\begin{align}
g(v) = \sum_{i=1}^m \left(C_i^{l-|v|} D_i^{|v|} +D_i^{l-|v|} C_i^{|v|} \right)  + \sum_{i=1}^s E_i^l,
\end{align}
for some positive constants $C_i,D_i$, $i \in [m]$, $E_i$, $i \in [s]$.
Moreover, using \eqref{pair}, 
\begin{align}
C^{l-|v|} D^{|v|} +D^{l-|v|} C^{|v|} \quad \stackrel{\cF}{\longleftrightarrow} \quad & (C+D)^{l-|w|} (C-D)^{|w|} + (C+D)^{l-|w|} (D-C)^{|w|}\\
\phantom{C^{l-|v|} D^{|v|} +D^{l-|v|} C^{|v|} } &= (C+D)^{l-|w|} (C-D)^{|w|} (1 + (-1)^{|w|}),
\end{align}
and $\cF(g)(w)$ has only positive coefficients since only the terms with $|w|$ even survive. Hence $\Gamma_l$ is convex when $k$ is even. 
\end{proof}

\end{document}